\documentclass[12pt]{amsart}
\usepackage{a4}
\usepackage{amssymb}
\usepackage{amsmath}
\usepackage{amsthm}
\usepackage{amstext}
\usepackage{amscd}
\usepackage{latexsym}
\usepackage{graphics}
\usepackage{color}
\newtheorem{thm}{Theorem}[section]
\newtheorem{lem}[thm]{Lemma}
\newtheorem{cor}[thm]{Corollary}
\newtheorem{prop}[thm]{Proposition}

\theoremstyle{definition}
\newtheorem{rmk}[thm]{Remark}
\newtheorem{defn}[thm]{Definition}

\newcommand{\tensor}{\otimes}

\newcommand{\intersection}{\cap}

\newcommand{\Ker}{{\rm Ker \ }}

\newcommand{\Spec}{{\rm Spec \,}}

\newcommand{\sC}{{\mathcal C}}

\newcommand{\sE}{{\mathcal E}}

\newcommand{\sL}{{\mathcal L}}

\newcommand{\sO}{{\mathcal O}}

\newcommand{\C}{{\mathbb C}}

\newcommand{\Z}{{\mathbb Z}}

\input{xy}
\xyoption{all}
\begin{document}
\title[Unitary representations of the fundamental group]{Unitary 
representations of the fundamental group of orbifolds}

\author[I. Biswas]{Indranil Biswas}

\address{School of Mathematics, Tata Institute of Fundamental
Research, Homi Bhabha Road, Bombay 400005, India}

%\email{indranil@math.tifr.res.in}

\author[A. Hogadi]{Amit Hogadi}

%\address{School of Mathematics, Tata Institute of Fundamental
%Research, Homi Bhabha Road, Bombay 400005, India}

%\email{amit@math.tifr.res.in}



\date{}

\begin{abstract}
There is a well known bijective correspondence between the
isomorphism classes of
polystable vector bundles $E$ with $c_i(E)\,=\,0$ for
$i\, \geq\, 1$ on a
smooth complex projective variety and the equivalence
classes of unitary representations of the fundamental group
of the variety. We show that this bijective correspondence
extends to smooth orbifolds.
\end{abstract}

\maketitle

\section{Introduction}

Let $Y/\C$ be a connected smooth projective curve, and let 
$\sE\,\longrightarrow \, Y$ be
a polystable vector bundle of degree zero. A celebrated theorem of 
Narasimhan and Seshadri, \cite{NS}, says that $\sE$ is necessarily given 
by a 
unitary representation of the
fundamental group of $Y$. Let $X/\C$ be a smooth projective variety
of dimension $n$.
Fix an ample line bundle $\mathcal L$ on $X$ in order to define
the degree of a torsionfree coherent sheaf on $X$. The following 
generalization of the Narasimhan--Seshadri theorem holds.

\begin{thm}[\cite{Do}, \cite{mr}]\label{mrns}
Let $\sE$ be a vector bundle on $X$ such that $c_1(\sE)\,=\,0$ and 
$c_2(\sE)\cdot c_1({\mathcal L})^{n-2}\,=\,0$. Then $\sE$ is polystable
if and only if it is given by a 
unitary representation of the (topological) fundamental group of $X$.
\end{thm}

Some clarification on Theorem \ref{mrns} is necessary.
By \cite[p. 231, Proposition 1]{Do}, a stable vector bundle on $X$ admits
a Hermitian--Yang--Mills connection. Since a polystable vector bundle $\sE$ on $X$
is a direct sum of stable vector bundles of same ${\rm degree}/{\rm rank}$
quotient, the Hermitian--Yang--Mills connection on the stable direct summands
together produce a Hermitian--Yang--Mills connection on $\sE$. If
$c_1(\sE)\,=\,0$ and
$c_2(\sE)\cdot c_1({\mathcal L})^{n-2}\,=\,0$, then any Hermitian--Yang--Mills
connection on $\sE$ is flat \cite[p. 115, Theorem 4.11]{Ko}. All Chern
classes in Theorem \ref{mrns} are topological, taking values in rational
cohomological classes.

Our aim here is to generalize the 
above theorem to projective orbifolds. By an {\bf orbifold} over a field 
$k$ we always mean a smooth separated Deligne--Mumford stack which is of 
finite type over $k$ and whose isotropy groups at all generic points is trivial.
Thus an orbifold has a dense open substack 
which is a scheme. We call a Deligne--Mumford stack projective if it has 
a coarse moduli space which is a projective variety. 

\begin{thm}\label{ns}
Let $X/\mathbb C$ be an irreducible projective orbifold of dimension
$n$, and let $\sL$ be an ample line bundle on $X$. Let 
$\sE\,\longrightarrow\, X$ be a vector bundle
on $X$ such that $c_1(\sE)\,=\, 0$ and $c_2(\sE)\cdot c_1({\mathcal 
L})^{n-2}\,=\, 0$. Then $\sE$ is polystable (with respect to
$\sL$) if and only if it is obtained by a unitary representation of 
$\pi_1^{top}(X,x)$ for a (equivalently, any) closed point $x\,\in\, X$.
\end{thm}

Here $\pi_1^{top}(X,x)$ denotes
the topological fundamental group of the underlying
complex analytical stack. See \eqref{ampledefn}
for definition of ample line bundles on a stack. All Chern
classes in Theorem \ref{ns} are topological, taking values in rational
cohomological classes. Theorem \ref{ns} follows
quite easily from Theorem \ref{mrns} in the 
special case when $X$ is a global quotient of a smooth variety by a 
finite group. The main work needed to prove Theorem \ref{ns}
involves deducing
the general case from this special case.

\section{Rigidification}

Let $k$ be a field. Throughout this section we work with 
Deligne--Mumford 
stacks over $k$, and assume that all these stacks are of finite type and 
separated over $k$ and are quotient stacks, meaning, can be expressed as 
quotient of an affine scheme by the action of a linear algebraic group. 
As shown in 
\cite[Proposition 5.2]{kresch}, such stacks, Zariski locally, 
can be expressed as 
quotient of an affine variety by a finite group. The inertia stack of a 
stack $Y$, viewed as a sheaf of groups on the big \'etale site of $Y$,
will be denoted by $I_Y$.

Recall that the notion of rigidification of a Deligne--Mumford stack is 
a 
process to get rid of given stabilizers in a ``minimal'' manner;
see \cite[Section 5.1]{rigid}. Throughout this
paper, we fix the following definition of rigidification. It
should be clarified that this definition differs from the one in
\cite[Section 5.1]{rigid}.

\begin{defn}
Let $f:Y\longrightarrow X$ be a $1$--morphism of
Deligne--Mumford stacks over $k$. This $f$ is called a {\bf 
rigidification} if 
for any atlas $U\,\longrightarrow\,X$, where $U$ is a scheme, 
the projection $Y\times_XU\,\longrightarrow\, U$ 
is a coarse moduli space. 
\end{defn}

Since the morphism to the coarse moduli space of
a separated Deligne--Mumford stack is always proper, it follows that a
rigidification morphism is a proper morphism.

As a direct consequence of the above definition, it follows that every 
morphism from a Deligne Mumford stack to its coarse moduli space is a 
rigidification. Proposition \ref{rspecial} provides additional 
examples of rigidification.

\begin{lem}\label{specialdiag}
Let $g:Y_1\longrightarrow Y_2$ be a $1$--morphism of Deligne--Mumford 
stacks over $k$. Then for $i=1,2$, there exists 
\begin{enumerate}
\item[(i)] affine schemes $U_i$ and linear algebraic groups $G_i/k$ 
acting on $U_i$,
\item[(ii)] a group homomorphism $\phi:G_1\longrightarrow G_2$, and 
\item[(iii)] a $k$--morphism $f:U_1\longrightarrow U_2$ which is 
equivariant with respect to $\phi$,
\end{enumerate}
such that $Y_i=[U_i/G_i]$, and the following diagram is commutative
$$\xymatrix{
U_1\ar[r]^f\ar[d] & U_2 \ar[d] \\
Y_1=[U_1/G_1] \ar[r] & [U_2/G_2]=Y_2 
}$$
\end{lem}

\begin{proof}
Take any $i\, \in\, [1\, ,2]$.
Since $Y_i$ is a quotient stack, we can write $Y_i=[V_i/H_i]$, where 
$V_i$ is an affine scheme and $H_i$ is a linear algebraic group acting on $V_i$. 
 Now we let 
$$\widetilde{V}_1=V_1\times_{Y_2}V_2= 
V_1\times_{Y_1}(V_2\times_{Y_2}Y_1)\, .$$
Then clearly $$[\widetilde{V}_1/(H_1\times H_2)]=[V_1/H_1]\, .$$
The projection 
$$ \widetilde{V}_1\longrightarrow V_1$$ is affine, because it
is a base extension of the affine morphism $V_2\longrightarrow Y_2$.

Since 
$V_1$ is an affine variety, and the projection 
$\widetilde{V}_1\longrightarrow V_1$ is 
affine, we conclude that $\widetilde{V}_1$ is also an affine 
variety. Also, we have a
natural map $$h:\widetilde{V}_1\longrightarrow V_2$$ which is 
equivariant for the actions
of $H_1\times H_2$ on $\widetilde{V}_1$ and $V_2$, with $H_1\times H_2$
acting on $V_2$ through the projection
$H_1\times H_2\longrightarrow H_2$.
Now, because of the identification $[\widetilde{V}_1/(H_1\times 
H_2)]=[V_1/H_1]$, the proof is completed by setting $G_1=H_1\times H_2$, 
$G_2=H_2$, $U_1=\widetilde{V}_1$ and $U_2=V_2$.
\end{proof}

The following lemma is a generalization of the universal property of
coarse moduli spaces.
\begin{lem}\label{factor}
Let $g:Y_1\longrightarrow Y_2$ be as in Lemma \ref{specialdiag}. Assume, 
that for 
every point $p$, the map induced by $f$ from the isotropy group of 
$Y_1$ at $p$ to that of $Y_2$ at $f(p)$ is trivial. Then $f$ factors 
through the coarse moduli space of $Y_1$.
\end{lem}

\begin{proof}
Let $U_i, G_i, \phi, f$ be as in Lemma \ref{specialdiag}. Let 
$K={\rm Kernel}(\phi)$. The assumption that the map induced by $f$ on 
all isotropy groups is trivial is equivalent to saying that the 
stabilizer of every point in $U_1$ is contained in $K$.

We claim that the geometric quotient $V=U_1/\!/K$ exists as an 
algebraic 
space. In order to prove this, first consider the stack $[U_1/K]$. 
It is a Deligne Mumford stack, because the isotropy groups 
of the action of $G_1$ (and hence of $K$) on $U_1$ are finite \'etale 
over $k$. Thus $V$, which is nothing but the coarse 
moduli space of $[U_1/K]$, exists as an algebraic
space by a theorem of Keel and Mori \cite[Theorem 1.1]{keelmori}.

Next, we will show that the action of $G_1/K$ on $V$ is free. To prove 
this, take
$\widetilde{x}\in V$ and $\overline{g}\in G_1/K$ such that 
$\overline{g}(\widetilde{x})=\widetilde{x}$. Choose a point $x\in U_1$ 
lying over $\widetilde{x}$, and also fix a lift $g\, \in\, G_1$ of 
$\overline{g}$. 
Since $V$ 
is a geometric quotient of $U_1$ by $K$, points of $V$ correspond to 
orbits in $U_1$ for the action of $K$. Thus, there exists an
element $a\in K$ such that 
$$ gx = ax \, .$$
Alternatively, $a^{-1}g$ is in the stabilizer of the point $x$ and hence 
is contained in $K$. Thus $g\in K$, which implies that 
$\overline{g}\,=\, eK$. This proves the above assertion that the action 
of $G_1/K$ on $V$ is free.

Consequently, the quotient stack $[V/(G_1/K)]$ is actually an algebraic 
space which 
is also the geometric quotient of $U_1$ by $G_1$, and hence 
$[V/(G_1/K)]$ is the coarse moduli space of $Y_1$.

Since the map $f:U_1\longrightarrow U_2$ intertwines the given action of 
$K$ on $U_1$ and the trivial action of $K$ on $U_2$, it factors 
through $V$. 
Thus we have an induced map $V\longrightarrow U_2$ which is equivariant
with respect to the action of $G_1/K$ on $V$ and the action of $G_2$ on 
$U_2$, using the homomorphism $\phi$ in Lemma \ref{specialdiag}.
This in turn induces a map $[V/(G_1/K)]\longrightarrow [U_2/G_2]$. But 
as explained above, 
$[V/(G_1/K)]$ is the coarse moduli space of $Y_1$. This finishes the proof.
\end{proof}

\begin{lem}\label{closed}
Let $f:Y\longrightarrow X$ be a $1$--morphism of separated 
Deligne--Mumford stacks over 
$k$. Let $\Sigma$ denote the set of all points in $Y$ where the induced 
map on isotropy groups is an isomorphism. Then $\Sigma$ is an open subset of $Y$.
\end{lem}

\begin{proof}
Since the Deligne--Mumford stacks are separated, the inertia stacks 
$I_Y$ 
and $I_X$ are finite over $Y$ and $X$ respectively. Thus the kernel and 
cokernel of the induced map $I_Y\longrightarrow f^{-1}(I_X)$ are finite 
over $X$. Hence the set of all points where both the kernel and cokernel 
are trivial is an open subset of $Y$.
\end{proof}

\begin{prop}\label{rspecial}
Let $U/k$ be a $k$-scheme, and let $G$ be a finite group acting on $U$. 
Let 
$K$ be a normal subgroup of $G$, and let $V$ denote the geometric 
quotient of $U$ by $K$. Then the following two hold:
\begin{enumerate}
 \item[(i)] The natural map $f:[U/G]\longrightarrow [V/(G/K)]$ is a 
rigidification.
\item[(ii)] If $U$ is spectrum of a local ring, then any rigidification 
of $[U/G]$ is of this type for some normal subgroup $K$ of $G$.
\end{enumerate}

\end{prop}
\begin{proof}$(i)$ \ Since $V\longrightarrow [V/(G/K)]$ is \'etale, and 
$[U/K]\longrightarrow V$ is 
a coarse moduli space, in order to prove the statement it is enough to 
show 
that the following diagram is Cartesian:
$$\xymatrix{
[U/K]\ar[r]^{f'}\ar[d]_g & V\ar[d]^{g'} \\
[U/G]\ar[r]^f & [V/(G/K)]
}$$

We first recall the description of the $1$-morphism $f$. 
The $k$-groupoid underlying the stack $[U/G]$ is the category of all 
triples $$\left( Z, P\stackrel{h}{\longrightarrow} 
Z,P\stackrel{\phi}{\longrightarrow} U 
\right)\, , $$ where $Z$ is a $k$ scheme, $h$ is a principal $G$-bundle 
and 
$\phi:P\longrightarrow U$ is a $G$-equivariant morphism. Given such a 
triple, one has the corresponding induced map
$$\phi'\,:P'\,=P/K\,\longrightarrow \,V\, ,$$ where 
$P/K\longrightarrow Z$ is now a principal
$G/K$-bundle. One thus gets a triple $(Z,P'\stackrel{h'}{\to}Z, \phi')$ 
which defines an object in the groupoid underlying $[V/(G/K)]$. This assignment 
$$ (Z,P,\phi)\longrightarrow (Z,P',\phi')$$ defines the $1$-morphism 
$f$. 

To prove that the above diagram is Cartesian, we need to exhibit a 
$1$-isomorphism $$[U/K]\longrightarrow [U/G]\times_{[V/(G/K)]}V\, .$$
By construction of fiber product of groupoids, the groupoid underlying 
$$[U/G]\times_{[V/(G/K)]}V$$ is the category of all
$$(Z,P\stackrel{h}{\longrightarrow}Z,\phi,Z
\stackrel{\gamma}{\longrightarrow} V,\theta)\, ,
$$
where $(Z,P\stackrel{h}{\longrightarrow}Z,\phi)$ is an 
object, say $\alpha$, of groupoid underlying $[U/G]$, and $\gamma$ 
defines 
an object, say $\beta$, of the groupoid underlying $V$, and $\theta$ is 
an isomorphism $f(\alpha)\longrightarrow g'(\beta)$. Note that by above 
discussion, $f(\alpha)$ is nothing but the object defined by the triple 
$$(Z,P',\phi')\, ,$$
where $P'$ is the principal $G/K$-bundle $P/K\longrightarrow Z$, and 
$\phi':P'\longrightarrow V$ 
is the induced morphism. Moreover, $g'(\beta)$ is nothing but the triple
$$(Z,Z\times (G/K), \widetilde{\gamma})\, ,$$
where $Z\times G/K\longrightarrow Z$ is the trivial principal 
$G/K$-bundle, 
and $$\gamma:Z\times G/K\longrightarrow V$$ is the $G/K$ equivariant map 
given by 
$$(z,g) \,\longmapsto\, g\cdot \gamma(z)\, .$$
Thus giving the isomorphism 
$$ \theta:(Z,P',\phi')\longrightarrow 
(Z,Z\times(G/K),\widetilde{\gamma})$$
is equivalent to giving a trivialization (or a section) of 
$P'\longrightarrow Z$ and 
imposing the condition that the induced map $\phi'$ coincides with 
$\widetilde{\gamma}$. Note that $\gamma:Z\longrightarrow V$ could have 
been any 
arbitrary $k$-morphism, to start with. Thus, the category underlying the 
groupoid $[U/G]\times_{[V/(G/K)]}V$ coincides with the category of 
all quadruples
$$(Z,P\stackrel{h}{\longrightarrow}Z,\phi,s)\, ,$$
where $Z,P$ and $\phi$ are as above, and $s$ is a trivialization of the 
principal $G/K$-bundle $h':P/K\longrightarrow Z$. But giving a 
trivialization of $h'$ 
is equivalent to giving a reduction of the structure group of $h$ to 
$K$. Thus the above category of quadruples is equivalent to the 
category of all triples
$$(Z,Q\stackrel{q}{\longrightarrow}Z,\psi)\, ,$$
where $q:Q\longrightarrow Z$ is a principal $K$-bundle and 
$\psi:Q\longrightarrow U$ is a 
$K$-equivariant morphism. One now sees that this is also the groupoid 
underlying $[U/K]$.

\noindent $(ii)$ Let $U$ be a spectrum of a local ring, and let $G$
be a finite group acting on $U$. Define $Y:=[U/G]$.
Let $$f\,:\,Y\,\longrightarrow\, X$$
be any rigidification map. Let $K$ be the kernel of
the homomorphism $I_p\longrightarrow 
I_{f(p)}$, where $I_p$ and $I_{f(p)}$ are the isotropy groups 
for the unique closed point 
$p$ of $Y$ and $f(p)$ respectively. Consider the $G$-invariant map 
$$g\,:\,U\,\longrightarrow\, X$$ 
induced by $f$. Let $V$ be the geometric quotient of $U$ by $K$.

We will first show that $g$ factors through $V$. Indeed, the 
induced map from the stack $[U/K]\longrightarrow X$ is trivial on all 
isotropy 
groups, and hence by Lemma \ref{factor}, it factors through the coarse 
moduli space, which is $V$.

Since $g$ is $G$-invariant, the induced map 
$V\longrightarrow X$ is $G/K$ invariant, and hence it gives us a map 
$$[V/(G/K)]\,\longrightarrow\, X\, .$$ Clearly, this map 
is also a rigidification, but it is also an isomorphism on the isotropy group at the closed 
point of the domain. Hence by Lemma \ref{closed}, it is an isomorphism 
on the isotropy 
group at all points. Therefore, this map must be an isomorphism. 
\end{proof}

\begin{lem} \label{uniqueness}
Let $f:Y\longrightarrow X$ be a rigidification, and 
let $g:Y\longrightarrow Z$ be any $1$-morphism. 
Suppose $g$, factors through $f$, meaning there exists a $1$-morphism 
$h:X\longrightarrow Z$ and a $2$-isomorphism $\alpha\, :\, h\circ f 
\, \Rightarrow\, g$. Then 
the pair $(h,\alpha)$ does not have any $2$-automorphisms, i.e., for
any $\theta\, :\, h\,\Rightarrow\, h$, the commutativity of the diagram
$$\xymatrix{
h\circ f \ar[d]_{\theta\circ f} \ar[r]^\alpha & g \ar@{=}[d] \\
h \circ f \ar[r]^{\alpha} & g
}$$
implies that $\theta$ is identity. 
\end{lem}

\begin{proof}
For an algebraically closed field $K$ over $k$, let $X_K$ denote the 
category of $K$-points of $X$. Thus $f$, $g$ and $h$ induce functors
$$ \xymatrix{
Y_K \ar[d]_{f_K} \ar[r]^{g_K} & Z_K \\
X_K \ar[ru]_{h_K}
}$$
and a natural equivalence $\alpha_K\,:\, h_K\circ f_K \,\Rightarrow\, 
g_K$. In order to prove that $\theta$ is identity, it is enough to show 
that 
$\theta_K$ is identity for every algebraically closed field $K$ over 
$k$. However, from the definition of rigidification it follows that the 
functors $f_K$ is full and essentially surjective. It is then a simple 
exercise in category theory to show that $\theta_K$ is identity.
\end{proof}

Let $f:Y\longrightarrow X$ be a rigidification. Let $I_Y$ 
(respectively, 
$I_X$) be the inertia sheaf of $Y$ (respectively, $X$), and 
define $$K^f\,:=\,\Ker(I_Y\longrightarrow 
f^{-1}(I_X))\, .$$ Then we claim that $f$ has the following universal 
property (which is a generalization of Lemma \ref{factor}):

\begin{prop}\label{universal}
Given any $1$-morphism $g:Y\longrightarrow Z$ such that $K^f$ is 
contained in the kernel of $I_Y\longrightarrow g^{-1}(I_Z)$, there 
exists a morphism $h:X\longrightarrow Z$, and a $2$-isomorphism 
$\alpha\, :\, h\circ f\,\Rightarrow \, g$. Further $(h,\alpha)$ is 
uniquely determined up to a unique $2$-isomorphism (see 
\eqref{uniqueness}). 
\end{prop}

\begin{proof} Uniqueness of the $2$-isomorphism follows from Lemma 
\ref{uniqueness}.
Because of the uniqueness claimed, it is enough to prove the proposition 
Zariski locally 
around a point $p \in Y$. Thus by Lemma \ref{specialdiag} we may 
reduce to the case where $Y=[U_1/G_1]$, $Z=[U_2/G_2]$, and $g$ is 
induced by a map $h:U_1\longrightarrow U_2$ which is equivariant with 
respect to a given group 
homomorphism $\phi:G_1\longrightarrow G_2$. Moreover, we may assume that 
$G_1$ is precisely the isotropy group at the point $p$, and $K^f\subset 
\Ker(\phi)$ is a normal subgroup of $G_1$. By Proposition \ref{rspecial}(ii), by further taking a smaller Zariski neighborhood of $p$, we may assume that the rigidification $f:Y\to X$ is of the form $[V/(G_1/K^f)]$ where $V$ is the geometric quotient of $U_1$ by $K^f$. Note that although Proposition \ref{rspecial}(ii) is stated in the case where $U_1$ is spectrum of a local ring, we use standard limiting argument to draw conclusion about a small enough Zariski neighborhood of $U_1$ (and hence also of $Y$). Since $K^f \subset \Ker(\phi)$, $h$ induces a map from $V\to U_2$ which is equivariant w.r.t. $G_1/K^f$ action on $V$ and $G_2$ action on $U_2$. Thus it induces a unique map $[V/(G_1/K^f)]\longrightarrow Z$. This proves the claim.
\end{proof}

We have the following corollary of Proposition \ref{universal}.

\begin{cor}
A rigidification $f:Y\longrightarrow X$ is determined uniquely by 
$$\Ker(I_Y\longrightarrow f^{-1}(I_X))\, .$$ 
\end{cor}

\section{Rigidification of Complex Deligne--Mumford stacks}

We continue using the notation and assumptions of the previous section, 
but now restrict ourselves to the case $k=\C$. Although our main 
interest is algebraic stacks, we will have to consider complex analytic
Deligne--Mumford stacks over $\C$ which arise as infinite covering 
spaces 
of algebraic stacks. These complex analytic stacks will have the 
property that of being, locally in complex analytic topology,
quotient of an analytic variety by a finite group. In this section we 
only consider complex analytic stacks which have this additional 
property. We have the analogous definition and results for complex 
analytic Deligne--Mumford stacks.

\begin{defn}
Let $Y$ be a complex analytic Deligne--Mumford stack. Then a 
rigidification of $Y$ is a $1$-morphism $f:Y\longrightarrow X$, which
locally on $X$ in analytic topology, is a coarse moduli space.
\end{defn}

Proof of the following universal property is similar to that of 
Proposition \ref{universal}, thanks to the assumption that all complex 
analytic stacks considered here are locally quotient spaces for finite 
groups actions.

\begin{prop}\label{universalc}
Let $f:Y\longrightarrow X$ be a rigidification of complex analytic 
stack $Y$. Let 
$K^f$ be the kernel of the homomorphism $I_Y\longrightarrow 
f^{-1}(I_X)$. Then given any $1$-morphism 
$g:Y\longrightarrow Z$ such that $K^f$ is in the kernel of 
$$I_Y\,\longrightarrow\, g^{-1}(I_Z)\, ,$$ 
the $1$-morphism $g$ 
factors uniquely (up to a $2$-morphism) through $f$. In particular the 
rigidification $f$ is itself uniquely determined by the sheaf $K^f$.
\end{prop}

\begin{lem}\label{etalec}
Let $Y$ be a Deligne--Mumford stack over $\C$, and 
let $h:Z\longrightarrow Y$ be a 
(possibly infinite) covering map. Let $f:Y\longrightarrow X$ be a 
rigidification, 
and $$K^f\,:=\,\Ker(I_Y\longrightarrow f^{-1}(I_X))\, .$$ 
Assume that $h^{-1}(K^f)$ lies 
in the image of the homomorphism $I_Z\longrightarrow h^{-1}(I_Y)$.
Then there exists a unique (up to isomorphism) covering 
map $X'\longrightarrow X$ such that $Z\,\cong\, X'\times_XY$.
\end{lem}

\begin{proof} 
Since $h$ is representable, the homomorphism $I_Z\longrightarrow 
h^{-1}(I_Y)$ is injective.
We first claim that there exists a rigidification $g:Z\longrightarrow 
Z'$ such that 
$$h^{-1}(K^f)=\Ker(I_Z\longrightarrow g^{-1}(I_{Z'}))\, .$$
By uniqueness in Proposition \ref{universalc}, 
we may construct the rigidification analytically locally on $Z$. Due to
Proposition \ref{rspecial}, we may assume that $Y=[U/G]$, where $G$ is a 
finite group acting on an affine variety $U$, and there exists a normal 
subgroup $H\subset G$ such that $X=[V/(G/H)]$, where $V$ is the 
geometric quotient of $U$ by $H$. Note that since the map $h$ is 
representable, $\widetilde{U}=U\times_YZ\longrightarrow U$ is 
representable, and hence 
$\widetilde{U}$ is a analytic space. Moreover, it has an induced action 
of $G$. Now it is easy to see that the $Z'=[\widetilde{V}/(G/H)]$, 
where $\widetilde{V}$ is 
the geometric quotient of $\widetilde{U}$ by $H$, satisfies the claim.
Finally, $X'\,=\, Z'$ satisfies the condition in the lemma.
\end{proof}

Let $Y$ be a complex analytic Deligne--Mumford stack over $\C$, and let 
$y$ be a 
$\C$-point of $Y$. Let $\pi_1^{top}(Y,y)$ denote the topological 
fundamental group of $Y$. Let $G_y$ be the isotropy group of the stack $Y$ at the point $y$. 
The stack $[\Spec(\C)/G_y]$ will be denoted by
$BG_y$; it is well known that $[\Spec(\C)/G_y]$ is the 
classifying space of the group $G_y$. Since $y$ is a complex point, the 
residual gerbe at $y$ is neutral (as $\C$ is algebraically closed!). Thus there exists a canonical morphism 
$$\eta_y: \Spec(\C)/[G_y] \longrightarrow Y\, .$$
This induces a map on the fundamental groups 
$$ \pi_1^{top}(BG_y,p) \longrightarrow \pi_1^{top}(Y,y)\, ,$$ 
where $p$ is the unique (up to a $2$-isomorphism) complex point of $BG_y$. 
However, $$\pi_1^{top}(BG_y,p)=G_y\, .$$
We thus get a canonical map
$$\phi_y:G_y\longrightarrow \pi_1^{top}(Y,y)$$
(see also (\cite[Section 7]{noohi}).
Moreover, if $y'$ is any other complex point of $Y$, then composing $\phi_{y'}$ with a choice of an inner automorphism 
$$\pi_1^{top}(Y,y)\longrightarrow \pi_1^{top}(Y,y')$$ gives us a map
$$ \phi_{yy'}:G_{y'}\longrightarrow \pi_1^{top}(Y,y)$$
which is well defined up to an inner automorphism of $\pi_1^{top}(Y,y)$.

We now set up some more notation. Let $f:Y\longrightarrow X$ be a 
rigidification of Deligne-Mumford stacks over $\C$. Consider $K^f$ 
defined as in the statement of Lemma \ref{etalec}.
For any $\C$-point $p$ of $Y$, let $K^f_p \subset G_p$ be the stalk of 
$K^f$ at $p$. Fix a point $y\in Y(\C)$, and let $x=f(y)$. Let 
$N^{top}(f)$ denote the normal subgroup of $\pi_1^{top}(Y,y)$ generated 
by all $\phi_{yy'}(K^f_{y'})$, where $y'$ runs over $\C$-points of 
$Y$. Note that 
since $\phi_{yy'}$ is well defined up to conjugation, the 
subgroup $N^{top}(f)$ is well defined. 

\begin{lem}\label{ntopkf}
Notation is as above. Let $h:Z\longrightarrow Y$ be a connected (not 
necessarily finite) 
Galois \'etale cover with Galois group $G$. Let $z\in Z$ be a $\C$-point 
which lies over $y$. Then the following conditions are equivalent.
\begin{enumerate}
 \item $N^{top}(f)$ is contained in the kernel of 
$\pi_1^{top}(Y,y)\longrightarrow G$.
 \item $h^{-1}(K^f)$ lies in the image of $I_Z\longrightarrow 
h^{-1}(I_Y)$.
\end{enumerate}
\end{lem}

\begin{proof}
Let $z'$ be any $\C$-point of $Z$, and let $y'=h(z')$. As explained 
above, we have a natural $1$-morphism 
$$ BG_{y'}\longrightarrow Y\, .$$
Since $h:Z\longrightarrow Y$ is a Galois \'etale cover with Galois group 
$G$, it follows that $Z_{y'}=Z\times_YBG_{y'}\to BG_{y'}$ is also a Galois \'etale 
cover (not necessarily connected). Here, by a non-connected Galois \'etale cover, we mean a cover such that the group of deck transformations acts transitively on all complex points. Since all connected Galois \'etale 
covers of $BG_{y'}$ are of 
the form $BH$ for a normal subgroup $H\subset G_{y'}$, we see that 
there exists a normal subgroup $H\subset G_{y'}$ such that $Z_{y'}$ is 
a disjoint union of copies of $BH$. Statement (1) in the lemma is 
equivalent to the assertion
that $K^f_{y'}$ is contained in the kernel of the composite map 
$$ G_{y'}=\pi_1^{top}(BG_{y'},y') \stackrel{\phi_{yy'}}{\longrightarrow} 
\pi_1^{top}(Y,y) \longrightarrow G$$
for all such $z'$.

But the kernel of the above composite map is precisely the subgroup $H$, 
mentioned above. 
Thus $K^f_{y'} \subset H$. Note that every point $p$ of $Z$ lying over 
$y'$ gives rise to a connected component of $Z\times_YBG_{y'}$. 
Moreover, this connected component is nothing but $BG_p$, where $G_p$ 
is the isotropy group of $Z$ at $p$. Since $Z\times_YBG_{y'}$ is 
disjoint union of copies of $BH$, we see that the isotropy group at 
every point $p$ of $Z$ lying over $y'$ contains $H$ and hence also 
contains $K^f_{y'}$. Since this holds for all points $y'$ of $Y$, it 
is equivalent to the statement 2 in the lemma. 
\end{proof}

The following theorem is influenced by \cite{noohi}, and is a 
straightforward generalization of results in \cite{noohi}.

\begin{thm}\label{pi1}
Let $f\,:\,Y\,\longrightarrow\, X$ be a proper 
morphism of Deligne--Mumford stacks over $\C$. Let $y$ be a 
$\C$-point of $Y$, and let $x\,=\,f(y)$. Assume that
there exists a dense open
substack $U\subset X$ such that the
codimension of $X\backslash U$ is at least two, and the induced morphism 
$$f\vert_U\,:\,f^{-1}(U)\,\longrightarrow\, U$$ is a rigidification. 
Then the following sequence of groups is exact
$$ 1\longrightarrow N^{top}(f)\longrightarrow \pi^{top}_1(Y,y)
\longrightarrow \pi_1^{top}(X,x)\longrightarrow 1\, .$$
\end{thm}

\begin{proof} Let us first consider the case where $f$ is a 
rigidification. Let $Et(Y)$ denote 
the category of covering spaces of $Y$. Let $\widetilde{\sC}$ denote 
the full subcategory of objects $Z\stackrel{h}{\longrightarrow} Y$ in 
$Et(Y)$ such that $h^{-1}(K^f)$ is contained in the image of the 
homomorphism $I_Z\longrightarrow h^{-1}(I_Y)$, where $K^f$ is
the kernel of the homomorphism $I_Y\longrightarrow f^{-1}(I_X)$. By definition of $N^{top}(f)$ and in view of Lemma \ref{ntopkf}, we need to show that the category $\widetilde{\sC}$ is equivalent to the category
$Et(X)$ of coverings of $X$. If $X'\longrightarrow X$ is an \'etale 
covering, then it is easy to see that $X'\times_X Y\longrightarrow Y$ is 
in $\widetilde{\sC}$.

Conversely, let $Z\longrightarrow Y$ be an \'etale 
covering in $\widetilde{\sC}$. Then by Lemma \ref{etalec} there is an 
\'etale covering $X'\longrightarrow X$ such that $Z\,=\,X'\times_XY$. 

Now consider the general case where $f:Y\longrightarrow X$ is a 
rigidification outside a closed substack of codimension at least two.
As above, let $Et(Y)$ denote the category of all 
coverings of $Y$, and let $\widetilde{\sC}$ be the full subcategory 
consisting of all coverings $Z\longrightarrow Y$ such that $h^{-1}(K^f)$ 
is contained in the image of $I_Y\longrightarrow f^{-1}(I_X)$. Let 
$U\subset X$ be the open subset of $X$ whose 
complement has codimension at least two and such that 
$f^{-1}(U)\longrightarrow U$ is 
a rigidification. Let $\widetilde{\sC}_U\subset Et(f^{-1}(U))$
denote the full subcategory consisting of all coverings 
$Z\longrightarrow f^{-1}(U)$ with the property that 
$$h^{-1}(K\vert_{f^{-1}(U)}) \,\subset\, {\rm Image}( I_{Z} 
\longrightarrow 
I_{f^{-1}(U)}) $$
(just like $\widetilde{\sC}\subset Et(Y)$).

We need to show that $\widetilde{\sC}$ is equivalent to $Et(X)$. It is 
clear that for any covering $X'\longrightarrow X$ in $Et(X)$, the 
fiber product $X'\times_XY\longrightarrow Y$ is 
an object of $\widetilde{\sC}$. To complete the proof we need to show 
that any $h:Z\longrightarrow Y$ in $\widetilde{C}$ comes from an object 
in $Et(X)$. Let 
$h_U:Z_U\longrightarrow f^{-1}(U)$ denote the restriction of $h$ to 
$f^{-1}(U)$. 
Since $f^{-1}(U)\longrightarrow U$ is a rigidification, by the above 
special case
we know that there exists an object $X'_U\longrightarrow U$ in $Et(U)$ 
such that 
$Z_U=X'_U\times_Uf^{-1}(U)$. Moreover, since the complement of 
$U\subset X$ has codimension at least two, the
morphism $X'_U\longrightarrow U$ extends to a 
covering $X'\longrightarrow X$. Now to finish the proof we observe that 
the two 
covering spaces $Z\longrightarrow Y$ and $X'\times_XY\longrightarrow Y$ 
agree on the dense open 
subset $f^{-1}(U)$ and hence are isomorphic. 
\end{proof}

\begin{thm}\label{rd}
Let $X/\mathbb C$ be any quasi-projective orbifold. Then there exists a 
proper $1$--morphism $\phi:Y\longrightarrow X$ such that
\begin{enumerate}
 \item $Y$ is an orbifold which is a finite global 
quotient, and
 \item there exists a dense open subset $V\subset X$ such that 
$\phi^{-1}(V)\longrightarrow V$ is a rigidification morphism, and the 
complement $X\setminus V$ has codimension at least two. \end{enumerate}
\end{thm} 

\begin{proof} Let $q:X\longrightarrow X'$ be the coarse moduli space of 
$X$. 
Let $D\subset X'$ be a (reduced) divisor such that $q$ is an 
isomorphism over $X'\backslash D$. Since $X'$ is normal, there exists an 
open subset $V \subset X'$ such that
\begin{enumerate}
 \item $D \intersection V$ is smooth, and
 \item $X'\backslash V$ has codimension at least two. 
\end{enumerate}
Let $m$ be a positive integer such that for every point $p$ of $X$, the 
order of the isotropy subgroup at $p$ divides $m$. Let $\sL$ be a 
sufficiently ample line bundle on $X'$, and let $D'$ be the zero locus 
of a section of 
$\sL^m\tensor\sO_{X'}(-D)$ such that $(D+D')\intersection V$ is smooth. Note 
that since $X'$ is not necessarily smooth, the divisor $D$ may not be 
Cartier. Hence the 
reflexive sheaf $\sO_{X'}(-D)$ may not be a line bundle. We replace $D$ by 
$D+D'$ without loss of generality and assume that there exists a line 
bundle $\sL$ on $X'$ such that $\sL^m \cong \sO_{X'}(D)$.

Let $T\,\subset\, \sL$ denote the cyclic covering of $X'$ defined by the 
section $D$ of $\sL^m$. There is a natural action of $G\, :=\, \Z/m$ on
$T$ given by the action of ${\mathbb G}_m$ on $\sL$; by the choice of 
the integer $m$, we have a rational map $T \dashrightarrow X$ 
defined over all codimension one points of $X$. 

Let $\Gamma$ be the closure of the graph of the rational map 
$T\dashrightarrow X$. We note that $\Gamma$ is a Deligne--Mumford stack 
(possibly 
singular). Since the algorithm for resolution of singularities in 
\cite[p. 782, Theorem 1.0.3]{wlodarczyk} commutes with \'etale 
base change, it also 
applies to Deligne--Mumford stacks. We thus use 
\cite[p. 782, Theorem 1.0.3]{wlodarczyk} and functorially resolve the 
singularities 
of $\Gamma$ to obtain a smooth Deligne--Mumford stack $T'$ and a proper 
birational map $T'\longrightarrow \Gamma$. By functoriality of the 
resolution, there
is a natural action of $G$ on $T'$ such that the map $T'\longrightarrow 
\Gamma$ is 
$G$--equivariant. Also, by construction, the induced $G$--equivariant 
map 
$T'\longrightarrow X$ is a morphism. Thus 
we have an induced morphism $\phi:Y \longrightarrow X$, where 
$Y=[T'/G]$. It is now straight--forward to check that the
map $\phi$ satisfies the required conditions in the theorem.
\end{proof}

\section{Polystable vector bundles on orbifolds}

In this section we prove Theorem \ref{ns}. Throughout this section we work 
over the field $k\,=\,\C$.

\begin{defn}\label{ampledefn}
Let $Z$ be a projective Deligne--Mumford stack over $\C$. A line bundle 
$\sL$ on $Z$ is called ample if some power of $\sL$ descends to an ample 
line bundle on the coarse moduli space of $Z$. (See also 
\cite[2.1]{matsuki-olsson}.)
\end{defn}

Given a very ample line bundle $L_0$ on a normal projective variety 
$M_0$, the \textit{degree} $\text{degree}_{L_0}(F_0)$
of a torsionfree coherent sheaf $F_0$ on $M_0$
with respect to $L_0$ is defined to be the degree of the restriction
of $F_0$ to the general complete intersection curve obtained by
intersection hyperplanes on $M_0$ from the complete linear system
for $L_0$.

Fix an ample line bundle $\sL$ on a projective Deligne--Mumford stack 
$Z$. Let $Z_0$ be the coarse moduli space of $Z$.
The degree of a torsionfree coherent sheaf $F$ on $Z$ with respect
to $\sL$ is defined to be
$$
\text{degree}(F)\,=\, 
\text{degree}_{\sL}(F) \,:=\, \frac{1}{m^{d-1}n}\text{degree}_{{\sL}^m}
((\det F)^{\otimes n})\, ,
$$
where $d\,=\, \dim Z_0$, the integer $m$ is such that ${\sL}^m$
descends to a very ample line bundle on $Z_0$, and $n$
is such that $(\det F)^{\otimes n}$ descends to $Z_0$. It is easy
to see that $\text{degree}(F)$ is independent of the choices
of $m$ and $n$, but it depends on $\sL$.

A torsionfree coherent sheaf $F$ on $Z$ is called \textit{stable} 
(respectively,
\textit{semistable}) if for all coherent subsheaves $F'\, \subset\, F$
with $0\, <\, \text{rank}(F') \, <\, \text{rank}(F)$,
$$
\frac{\text{degree}(F')}{\text{rank}(F')}\, <\,
\frac{\text{degree}(F)}{\text{rank}(F)} ~\,~\,\text{(respectively,~}~
\frac{\text{degree}(F')}{\text{rank}(F')}\, \leq\,
\frac{\text{degree}(F)}{\text{rank}(F)}{\rm )}\, .
$$
A semistable sheaf is called \textit{polystable} if it is a direct
sum of stable sheaves.

\subsection{Normalization}\label{normalisation}
Let $X$ be an integral variety and $\Spec(L)\longrightarrow X$ be a 
dominant 
morphism, where $L$ is a field which is a finite extension of the 
function field of $X$. Then one has the notion of normalization of $X$ 
in $L$. One can easily extend this notion to the case when $L$ is 
replaced by a product of fields where the morphism is now dominant when 
restricted to each of the component fields. Now, if $\widetilde{X}$ 
denotes the normalization of $X$ in $L$, and $U\longrightarrow X$ is an 
\'etale morphism, then $U\times_X\Spec(L)$ is a product of fields with 
the map $$U\times_X\Spec(L)\longrightarrow U$$
being dominant on each component of the domain. Let 
$\widetilde{U}$ denote the normalization of $U$ in $U\times_X\Spec(L)$. 
Then the following diagram is Cartesian:
$$\xymatrix{
\widetilde{U}\ar[r]\ar[d] & \widetilde{X} \ar[d] \\
U \ar[r] & X
}$$
In other words, normalization commutes with \'etale (or even smooth) 
base change.

Thus, one can extend the notion to Deligne--Mumford stacks. In 
other words, if $X$ is a Deligne--Mumford stack, and 
$\Spec(L)\longrightarrow X$ is a 
dominant morphism from a field, one can define $\widetilde{X}$, the 
normalization of $X$ in $L$.

Concretely, to define $\widetilde{X}$, one 
first chooses an \'etale atlas $U\longrightarrow X$. Define 
$R\,:=\, U\times_XU$, and denote
by $\widetilde{U}$ (respectively, $\widetilde{R}$) the normalization of 
$U$ (respectively, $R$) 
in $U\times_X\Spec(L)$ (respectively, $R\times_X\Spec(L)$). Since 
$R\rightrightarrows U$ is an \'etale groupoid, so is $\widetilde{R} 
\rightrightarrows \widetilde{U}$. We leave these details to the reader, 
since they are simple exercise dealing with normalization of varieties. 
One then defines $\widetilde{X}$ as the quotient 
$[\widetilde{U}/\widetilde{R}]$. This 
construction 
can be seen more clearly when $X$ is quotient of a scheme $U$ by a 
finite group $G$. In this case one simply has a natural $G$ action on 
$\widetilde{U}$, and hence one defines $\widetilde{X}$ as 
$[\widetilde{U}/G]$. 

\subsection{Proof of Theorem \ref{ns}}

We first establish some ingredients of the proof.

\begin{lem}\label{des.}
Let $X/\mathbb C$ be a normal orbifold, and let $K$ be its function 
field. Let $f:Y\longrightarrow X$ be a dominant finite morphism, with 
$Y$ normal, such that the corresponding function field extension is
Galois. Let $G$ be this Galois group acting on $Y$. Let $V$
be a vector bundle on $X$, and let $W$ be a $G$--invariant
reflexive subsheaf of $f^*V$. Then there is a closed subset $S\,
\subset\,Y$ of codimension at least two such that $W\vert_{Y\setminus 
S}$ is a pullback of a unique subsheaf of $V\vert_{X\setminus f(S)}$.
\end{lem}

\begin{proof}
First assume that $X$ is a variety. Outside a closed subset
$S\, \subset\, Y$ of codimension at least two, the map $f$ is flat, and 
$W$ is locally free. The isotropies for the action of $G$ on $Y$
act trivially on the fibers of $f^*V$, and hence on the
fibers of $W$. Therefore, $W$ descends outside $f(S)$ to a unique
subbundle (see \cite[Theorem 1.2]{Ne} for descent of sheaves).

In the general case, let $U\, \longrightarrow\, X$ be an
atlas. From the uniqueness it is enough to prove by
replacing $X$ and $Y$ by $U$ and $Y\times_X U$ respectively,
and also replacing the sheaves by the corresponding pullbacks.
Then it is reduced to the above case.
\end{proof}

\begin{lem}\label{lemfinite}
Let $X/\mathbb C$ be a projective normal orbifold, and let $\sL$ be an 
ample line bundle 
on $X$. Let $f:Y\longrightarrow X$ be a dominant finite morphism, where
$Y/\mathbb C$ is 
a normal projective variety. Let $\sE$ be a vector bundle on $X$. Then 
the following two hold:
\begin{enumerate}
 \item[(i)] The vector bundle $\sE$ is semistable with respect to $\sL$ 
if and only if $f^*\sE$ is semistable with respect to $f^*(\sL)$.

\item[(ii)] If $\sE$ is polystable with respect to $\sL$, then 
$f^*\sE$ is polystable with respect to $f^*(\sL)$.
\end{enumerate}
\end{lem}

\begin{proof}
After one understands the notion of normalization of a Deligne--Mumford 
stack in a field (see \S~\ref{normalisation}), the proof is 
similar to the case when $X$ is a variety instead of an orbifold.
If $f^*\sE$ is semistable with respect to $f^*\sL$, then
clearly $\sE$ is semistable with respect to $\sL$. To prove the 
converse, assume that $\sE$ is semistable with respect to $\sL$.
We may replace $Y$ by its normalization in the Galois closure of the 
function field of $Y$ over that of $X$. If $f^*\sE$ is not semistable,
take the first term $F$ in the Harder--Narasimhan filtration of
$f^*\sE$. The descent of $F$ (see Lemma \ref{des.}) contradicts
semistability of $\sE$. Therefore, $f^*\sE$ is semistable.

Now assume that $\sE$ is polystable with respect to $\sL$. Hence
$f^*\sE$ is semistable with respect to $f^*\sL$ (by
part (i)). Consider the
socle $F\, \subset\, f^*\sE$; it is the unique maximal polystable
subsheaf with $\text{degree}(F)/\text{rank}(F) \,=\,
\text{degree}(f^*\sE)/\text{rank}(f^*\sE)$ (see \cite[p. 23, Lemma 
1.5.5]{huybrechts}). From the uniqueness of $F$ it follows that $F$
is a pullback of a subsheaf $F'$ of $\sE$ outside codimension
two (Lemma \ref{des.}). Since $\sE$ is polystable,
$F'\, \subset\, \sE$ has a direct summand $F''$. If $F\, \not=\,
f^*\sE$, the direct sum of $F$ with the socle of $f^*F''$ contradicts
the maximality of $F$. Hence $F\, =\, f^*\sE$, implying that
$f^*\sE$ is polystable.
\end{proof}

Let $X$ and $\mathcal L$ be as above. Let $d$ be the dimension of
$X$. All Chern classes considered here are topological, taking
values in rational cohomology classes.

\begin{prop}\label{prop-a}
Let $f\,:\,Y\,\longrightarrow \,X$ be any dominant finite
morphism of projective orbifolds. Fix an ample line bundle ${\mathcal 
L}$ on $X$. So $f^* {\mathcal L}$ is ample on $Y$. Let 
$\sE$ be a polystable vector bundle on $X$ with
respect to ${\mathcal L}$ such that 
$c_1(\sE)\,=\,0$, and $c_2(\sE)\cdot c_1({\mathcal L})^{d-2}\,=\,0$.
Then $f^*\sE$ is a polystable vector bundle on $Y$ with respect
to $f^* {\mathcal L}$. Also, $c_1(f^*\sE)\,=\,0$, and $c_2(f^*\sE)
\cdot c_1(f^*{\mathcal L})^{d-2}\,=\,0$.
\end{prop}

\begin{proof}
Clearly, $c_1(f^*\sE)\,=\,0$, and $c_2(f^*\sE)\cdot c_1(f^*{\mathcal
L})^{d-2}\,=\,0$, because these conditions hold $\sE$ and $\mathcal L$.

By \cite{vistoli-kresch} we can find a diagram $$\xymatrix{ 
Y'\ar[r]^h\ar[d]^{f'} & Y \ar[d]^f \\ X' \ar[r]^g & X }$$ where $Y'$ and 
$X'$ are smooth projective varieties, and the horizontal morphisms are 
finite and dominant. By Lemma \ref{lemfinite}, the pullback $g^*\sE$ is 
polystable, and also, $c_1(g^*\sE)\,=\,0$, and $c_2(g^*\sE)\cdot 
c_1(g^*{\mathcal L})^{d-2}\,=\,0$; note that $g^*{\mathcal L}$ is ample
because $g$ is finite. Thus, $g^*\sE$ is given by a unitary 
representation of the fundamental group of $X'$
(see Theorem \ref{mrns} and the paragraph following it). 
Therefore, $f'^*g^*\sE$ is also given by a unitary representation of the 
fundamental 
group of $Y'$, and hence it is polystable. Thus by Lemma \ref{lemfinite}, 
the pullback $f^*\sE$ is polystable.
\end{proof}

\begin{rmk}\label{rk1}
We first observe that proving Theorem \ref{ns} is equivalent to showing 
that if $\sE\,\longrightarrow\, X$ is polystable, and 
$\pi_X:U_X\longrightarrow X$ is the 
universal covering space of $X$, then 
\begin{enumerate}
 \item $\pi_X^*(\sE)$ is trivial, and
 \item $\pi_X^*(\sE)$ admits a $\pi_1^{top}(X,x)$ invariant
flat unitary metric. 
\end{enumerate}
\end{rmk}

\begin{lem}\label{global}
Let $X$ be a finite global quotient. Then
Theorem \ref{ns} holds for $X$.
\end{lem}

\begin{proof}
Take $\sE\,\longrightarrow\, X$ as in Theorem \ref{ns}. First
assume that $\sE$ is given by a unitary 
representation $\rho$ of $\pi_1^{top}(X,x)$. To prove that $\sE$
is polystable, first note that for any vector bundle $F$ on
$X$ given by a unitary representation of $\pi_1^{top}(X,x)$,
we have $c_1(F)\,=\, 0\,=\, c_2(F)$. Also, any unitary representation
is a direct sum of irreducible unitary representations. Therefore, it is
enough to prove polystability under the assumption that $\rho$ is
irreducible. Assume that $\rho$ is irreducible.

By hypothesis, there exists a smooth projective variety $Y$ and a finite
group $G$ acting on $Y$ such that $X\,=\,[Y/G]$. Let $$f\, :\,
Y\,\longrightarrow\, X$$ be the natural map. Since $f^*\sE$ is
given by the unitary representation $f^*\rho$, it follows that
$f^*\sE$ is polystable with respect to any polarization on
$Y$, in particular with respect to $f^*\mathcal L$ \cite[pp. 
177--178, Theorem 8.3]{Ko}.

We note that the vector bundle $\sE$ is polystable
because $f^*\sE$ is so (see Lemma \ref{lemfinite}).
To prove that $\sE$ is stable,
let $F\, \subset\, \sE$ be a nonzero subsheaf such that
$\text{rank}(F) \, <\, \text{rank}(\sE)$, and $\text{degree}(F)
\,=\, 0$. Then $\text{degree}(f^*F) \,=\, 0$. This implies
that $f^*F$ generates a subbundle of $f^*\sE$ which is
given by some some subrepresentation of $f^*\rho$ (see the last
paragraph of page 178 in the proof of \cite[Theorem 8.3]{Ko}).
Since $f^*F$ is a pullback from $X$, this contradicts the
irreducibility of $\rho$. Hence $\sE$ is stable.

To prove the converse, assume that $\sE\,\longrightarrow\, X$ is a 
polystable vector bundle with $c_1(\sE)\,=\, 0\,=\, c_2(\sE)\cdot 
c_1({\mathcal L})^{n-2}$. Take $Y$ as above such that $X\,=\, [Y/G]$.
Then $f^*\sE$ is a polystable vector bundle on $Y$ with $c_1(f^*\sE)
\,=\,0$ and $c_2(f^*\sE)\cdot c_1(f^*\sL)^{n-2}=0$, where $f$ is the
quotient map. Let $$\pi_Y:U_Y\longrightarrow Y$$ denote the universal 
covering space of $Y$. Then $\pi_Y^*f^*\sE$ is a trivial vector bundle 
on $U_Y$ with a $\pi_1^{top}(Y,y)$--invariant unitary flat metric, say 
$\langle -,- \rangle$. Note that $U_Y$ is also the covering space for 
$X$. Moreover, if $x=f(y)$, then $\pi_1^{top}(Y,y)$ is naturally a 
finite index subgroup of $\pi_1^{top}(X,x)$. Let $\{g_i\}$ denote its 
finite coset representatives. It is then clear that the ``averaged'' 
metric $$(v_1,v_2)\,:=\,\sum_ig_i^*\langle v_1,v_2\rangle$$ is a 
$\pi_1^{top}(X,x)$ invariant unitary flat metric on $U_Y$.
\end{proof}

We now reduce the general case of Theorem \ref{ns} to the above special 
case using Theorem \ref{rd}.

\begin{proof}[Proof of Theorem \ref{ns}]
As in Lemma \ref{global}, if $\sE$ is given by a
unitary representation
of $\pi_1^{top}(X,x)$, then $\sE$ is polystable,
and $c_1(\sE)\,=\, 0$ and $c_2(\sE)\cdot c_1({\mathcal
L})^{n-2}\,=\, 0$.

Assume that $\sE$ is polystable, 
and $c_1(\sE)\,=\, 0 $ and $c_2(\sE)\cdot c_1({\mathcal
L})^{n-2}\,=\, 0$.
Let $\phi:Y\longrightarrow X$ be a morphism as guaranteed by Theorem 
\ref{rd}. Fix a point $y_0$ of $Y$. Then, 
since $Y$ is a finite global quotient, by Lemma \ref{global}
and Proposition \ref{prop-a}, the pullback
$\phi^*\sE$ is given by a unitary representation of 
$\pi_1^{top}(Y,y_0)$. 

However, for any point $y\in Y$, the kernel of $G_y\longrightarrow 
G_{\phi(y)}$ acts trivially on the fiber of $\phi^*(\sE)$ at $y$. Thus 
this kernel is also contained in the kernel of the unitary representation. 
Then by Theorem \ref{pi1}, the unitary representation of 
$\pi_1^{top}(Y,y_0)$ which defines $\phi^*\sE$ actually factors through 
$\pi_1^{top}(X,\phi(y_0))$. This completes the proof
of Theorem \ref{ns}.
\end{proof}

As a consequence of 
Theorem \ref{ns}, we obtain the following generalization of Theorem 
\ref{mrns} to certain types of singular varieties.

\begin{cor}
Let $X/\C$ be a projective variety with at worst quotient singularities. 
Assume $X$ is the coarse moduli space of a projective orbifold 
$\widetilde{X}/\C$ (e.g. $X$ has isolated singular points). Then Theorem 
\ref{mrns} holds for $X$.
\end{cor}

\begin{proof}
Let $\sL$ be an ample line bundle on $X$ and $\sE$ be a polystable 
vector bundle satisfying $c_1(\sE)=0$ and $c_2(\sE)\cdot  c_1(\sL
)^{n-2}\,=\,0$, where $n$ is the dimension of $X$. Let $\pi\,:\,
\widetilde{X}\longrightarrow X$ be the given morphism. Since the 
pullback $\pi^*\sE$ is polystable (see Lemma \ref{lemfinite}), by 
Theorem \ref{ns}, the vector bundle $\pi^*\sE$ is given by a unitary 
representation of $\pi_1^{top}(\widetilde{X},x)$ (for some geometric 
point $x$ of $\widetilde{X}$). However, since the vector bundle arising 
from this representation is a pull back from $X$, it is clear that for 
every point $y$ of $\widetilde{X}$, the isotropy group $G_y$ at $y$ lies 
in the kernel of this representation. Thus by using \cite[p. 90, 
\S~8.1]{noohi} this descends to give a unitary representation of 
$\pi_1^{top}(X,x)$, which defines the vector bundle $\sE$.
\end{proof}

%%%%%%%%%%%%%%%%%%%%%%%%%%%%%%%%%%%%%%%%%%%%%%%%%%%%%%%%%%%%%%%%%%%%%%%%

\end{document}